\documentclass[reqno,12pt]{amsart}
\pdfoutput=1
\usepackage{amsmath,amsthm,amsfonts,amssymb,ifpdf}
\usepackage{amssymb}
\usepackage{amsmath}
\usepackage{amsthm}
\usepackage{graphicx}
\usepackage[all]{xy}
\usepackage{enumerate}
\usepackage{tikz-cd}

\theoremstyle{plain} 
\newtheorem{theorem}{Theorem}[section]
\newtheorem{proposition}[theorem]{Proposition}
\newtheorem{lemma}[theorem]{Lemma}
\newtheorem{corollary}[theorem]{Corollary}
\newtheorem{question}[theorem]{Question}
\newtheorem{example}[theorem]{Example}

\theoremstyle{definition} \newtheorem{definition}[theorem]{Definition}

\theoremstyle{remark} \newtheorem{remark}[theorem]{Remark}

\usepackage[all]{xy}

\ifpdf
  \usepackage[
    pdftex,
    colorlinks,%
    linkcolor=blue,citecolor=red,urlcolor=blue,
    hyperindex,%
    plainpages=false,%
    bookmarksopen,%
    bookmarksnumbered%
  ]{hyperref} 
 
 \usepackage{thumbpdf}
\else
  \usepackage{hyperref}
\fi

\newcommand\nnfootnote[1]{%
  \begin{NoHyper}
  \renewcommand\thefootnote{}\footnote{#1}%
  \addtocounter{footnote}{-1}%
  \end{NoHyper}
}

\newcommand{\CC}{\mathbb{C}}

\newcommand{\QQ}{\mathbb{Q}}

\newcommand{\ZZ}{\mathbb{Z}}
\def\Id{{\rm Id}}
\def\Tr{{\rm Tr}}
\def\diag{{\rm diag}}
\def\End{{\rm End}}

\newcommand{\ncom}{\newcommand}
\ncom{\mylabel}[1]{{\rm (#1)}\label{#1}}
\ncom{\Hom}{{\textit{Hom}}}
\ncom{\eop}{{\hfill $\Box$}}
\begin{document}
\baselineskip=16pt

\nnfootnote{Mathematics Classification Number: 15A15, 16U40, 16U60.}
\nnfootnote{Keywords:  Commutators, trace, idempotents, matrix ring, units.}

\setcounter{tocdepth}{1}

\title[Power of commutator being identity]{Commutators of finite multiplicative order}

\author{Arijit Mukherjee}
\address{Department of Mathematics\\ Indian Institute of Technology Madras\\ Tamil Nadu - 600 036, India.}
\email{mukherjee7.arijit@gmail.com}

\author{Gobinda Sau}
\address{Department of Mathematics\\ Indian Institute of Technology Kanpur\\ Uttar Pradesh - 208 016, India.}
\email{gobindasau.analytica@gmail.com}

\author{Arindam Sutradhar}
\address{Theoretical Statistics and Mathematics Unit\\ Indian Statistical Institute Bangalore\\ Karnataka - 560059, India.}
\email{arindam1050@gmail.com}

\begin{abstract}
This article studies the equation $[A,B]^k = \Id_n$ for matrices over $\CC$,
characterizing the pairs $(k,n)$ for which solutions exist via a classical result
of Lam and Leung on sums of roots of unity. The problem is next generalized to
matrix rings $M_n(S)$ over arbitrary unital rings $S$, where a sufficient condition
on the unity of $S$ is established and explicit constructions of solutions are provided. Beyond matrix rings, the structural implications of the equation $[a,b]^n = 1$ in a
general unital ring $R$ are investigated, yielding a collection of idempotents whose
properties govern the ring's structure. We prove that under a suitable condition on these idempotents,
$[a,b]^n = 1$ implies $R$ is isomorphic to $M_n(S)$ for some unital ring $S$.  We also provide an alternative proof using a result on characterisation of matrix rings by Goyal and Khurana. These
results together establish a framework connecting commutator equations and classical criteria for recognizing full matrix rings.
\end{abstract}
\maketitle
  


%

\section{Introduction}
 A classical fact of
matrix theory is that every commutator is traceless, and conversely every trace-zero matrix arises as a commutator. In particular,
the identity matrix can never be a commutator.  Therefore, it is natural to ask when it can be realized as a higher power of one : Given positive integers $k, n$, what are the matrices $A,B \in M_n(\CC)$ such that $[A,B]^k=\Id_n$?, where $[A,B]:=AB-BA$ denotes the \textit{commutator} of $A,B$ and $\Id_n$ is the identity matrix.
The relation $C^k = \Id_n$ means that $C=[A,B]$ is diagonalizable with distinct eigenvalues among the $k$-th roots of unity, so the problem reduces to determining when $n$ such roots of unity sum to zero - precisely the vanishing of $\Tr([A,B])$. This
arithmetic of vanishing sums of roots of unity was completely settled in a celebrated
article by Lam and Leung (cf. \cite{LaLe}).  Their result gives a clean arithmetic characterization of those $k$ and $n$ for which such matrices exist (cf. Theorem \ref{Main theorem_existential}).

We then generalize this question to matrix rings over an arbitrary unital ring $S$. Given a positive integer $n$, one can find elements $A,B \in M_n(S)$ such that $[A,B]^n=\Id_n$? This is a natural generalization, as the ring $S$ need not be a field. We show that a sufficient condition is that unity  $1$ of $S$ can be written as the sum of $n-1$ central units of $S$. Under this condition, we explicitly construct elements $A,B \in M_n(S)$ satisfying $[A,B]^n=\Id_n$ (cf. Theorem \ref{thm:main}).

We then move beyond the matrix rings and address the same question in the setting of a general unital ring. It is classical that the existence of elements of $a,b$ in a unital ring $R$, $1$ being the unity of $R$, satisfying $[a,b]=1$ forces a copy of Weyl algebra (cf. \cite[Definition 3, p.~281]{Me}) to embed in $R$.  This suggests that the equation $[a,b]=1$ is a strong structural constraint on $R$. We investigate an analogous question: what does the existence of $a,b \in R$ with $[a,b]^n=1$ imply about the structure of $R$\;? This is a more delicate question. Since $[a,b]^n=1$ is a weaker condition than $[a,b]=1$ and the structure of $R$ need not be as rigid. Nevertheless, meaningful structural conditions can still be drawn.

A fundamental result in this direction is due to Khurana and Lam (cf. \cite{KhLa}), who proved that if there exist idempotents $e_1,e_2$ in a unital ring $R$ such that $[e_1,e_2]$ belong to the set of invertibles $U(R)$, then $R \cong M_2(S)$ for some unital ring $S$. 
This is a striking result as it shows existence of a single commutator relation among idempotents can force the entire ring to be a full matrix ring. We establish a structural theorem for an arbitrary ring $R$: the existence of elements $a,b \in R$ satisfying $[a,b]^n
=1$ yields a collection of idempotents (cf. Lemma \ref{lem: Creating_idempotents}), and under a suitable condition on these idempotents, we conclude that $R \cong M_n(S)$, for some unital ring $S$ (cf. Theorem \ref{thm: Structure theorem due to commutator}). This provides a new criterion for recognizing when a ring is isomorphic to a full matrix ring, and highlights a connection between commutator equations and the classical theory of matrix ring characterizations.  We also provide an alternative proof of Theorem \ref{thm: Structure theorem due to commutator} using a result from the paper entitled ``A characterisation of matrix rings" by Goyal and Khurana (cf. \cite{GoKh}). While our final results are ring-theoretic, the engine throughout is linear-algebraic — the spectral structure of a commutator and the idempotents it generates and it is this thread that unifies the three settings in three upcoming sections. 
\section{Commutators of finite multiplicative order in $M_n(\mathbb{C})$}
It is well known that commutators in $M_n(\mathbb{C})$, the ring of $n\times n$ matrices with complex entries, are synonymous with trace zero matrices.  Therefore, it is obvious that the identity matrix $\Id_n\in M_n(\mathbb{C})$ can never be a commutator.  On that account, it is natural to ask whether $\Id_n$ can be realized as any higher power of a commutator or not.  We precisely ask the following :
\begin{question}\label{Main question_Section 2}
Let $k,n$ be two positive integers greater than $1$.  Let $M_{k,n}$ be the following subset of $M_n(\mathbb{C})\times M_n(\mathbb{C})$ :
$$M_{k,n}:=\big\{(A,B)\in M_n(\mathbb{C})\times M_n(\mathbb{C}) \mid [A,B]^k=\Id_n\big\}.$$ 
Then, for what values of $n$ and $k$, is the set $M_{k,n}$ non-empty?
\end{question}

Before answering this in complete generality, we give some instances where $M_{k,n}$ is non-empty by constructing some $A,B\in M_n(\mathbb{C})$ such that $(A,B)\in M_{k,n}$.  For a matrix $C$, we denote its trace by $\Tr(C)$.  Now, as for any $C\in M_n(\mathbb{C})$ with $\Tr(C)=0$, many existential and constructive proofs are available in the literature (cf. \cite{Sh} and \cite{AlMa}) for finding $A,B\in M_n(\mathbb{C})$ such that $[A,B]=C$, it is enough to produce some $C\in M_n(\mathbb{C})$ such that $\Tr(C)=0$ and $C^k=\Id_n$.
\begin{example}\label{Example 1_section 2}
\textbf{Both $n$ and $k$ are even}:  The diagonal matrix $C=\diag(1,-1,\cdots,1,-1)$ of size $n$ satisfies $\Tr(C)=0$ and $C^k=\Id_n$.
\end{example}
\begin{example}\label{Example 2_section 2}
\textbf{$k$ divides $n$}:  Let  $\zeta=e^{\tfrac{2\pi \mathrm{i}}{k}}$, where $\mathrm{i}\in \mathbb{C}$ is the imaginary unit satisfying $\mathrm{i}^2=-1$.  Then, the diagonal matrix $C_k=\diag(1,\zeta,\zeta^2,\dots,\zeta^{k-1})$ of size $n$ satisfies $C_k^k=\Id_k$ and $\Tr(C_k)=\sum_{i=0}^{k-1}\zeta^i=0$.  Now, as $n=kt$ for some positive integer $t$, 
the matrix $C\in M_{n}(\mathbb{C})$ given as
\begin{equation*}
    C=\diag(C_k,C_k,\dots,C_k) \;\;(t\; \text{many blocks})
\end{equation*}
satisfies the equalities $\Tr(C)=0$ and $C^k=\Id_{n}$.
\end{example}

We now provide a complete answer to Question \ref{Main question_Section 2}.  Let $m$ be any positive integer.  By $\mathcal{C}^m$ we denote the cyclic subgroup of $\mathbb{C}^{\ast}$ of all $m$-th roots of unity.  In \cite{LaLe}, the authors studied all possible values of $n$ such that there exist $n$ many $m$-th roots of unity that add up to zero.  The authors showed that the \textit{weight set} $W(m)$ \textit{of} $m$, defined as 
\begin{equation*}
    W(m):=\Big\{n\in \mathbb{Z}^+\mid \text{there exist\;}n \text{\;many\;} \eta_i\in \mathcal{C}^m \text{\;such that\;}\sum_{i=1}^n\eta_i=0 \Big\},
\end{equation*}
is obtained by taking all possible $\mathbb{N}$-combinations of prime divisors of $m$.  Precisely, \begin{proposition}[Main Theorem, p.~ 92, \cite{LaLe}]\label{Lam-Leung}
 Let $m=p_1^{\alpha_1}p_2^{\alpha_2}\cdots p_r^{\alpha_r}$ be a positive integer.  Then the weight set $W(m)$ is given as follows :
 $$W(m)=\mathbb{N}p_1+\mathbb{N}p_2+\cdots +\mathbb{N}p_r.$$
\end{proposition}
Translating this result into the language of commutators, we directly get an existential answer to the Question \ref{Main question_Section 2}.  
\begin{theorem}\label{Main theorem_existential}
Let $k$ be a positive integer and  $k=p_1^{\alpha_1}p_2^{\alpha_2}\cdots p_r^{\alpha_r}$ be its prime factorization.  Then $M_{k,n}$ is non-empty if and only if $n\in \mathbb{N}p_1+\mathbb{N}p_2+\cdots +\mathbb{N}p_r$.
\end{theorem}
\begin{proof}
For any matrix $C\in M_n(\mathbb{C})$, $C^k=\Id_n$ if and only if the minimal polynomial $p_C$ of $C$ is a factor of $x^k-1$.  Therefore, there exist $A,B\in M_n(\mathbb{C})$ with $C=[A,B]$ and $(A,B)\in M_{k,n}$ if and only if there exists $\zeta_i \in \mathcal{C}^k$ with $p_C(\zeta_i)=0$, $1\leq i \leq n$, such that $\sum_{i=1}^n \zeta_i =0$.  Equivalently, $n$ must lie in the weight set $W(k)$ of $k$.  That is, $n\in \mathbb{N}p_1+\mathbb{N}p_1+\cdots +\mathbb{N}p_r$ by Proposition \ref{Lam-Leung}.  
\end{proof}
\section{Commutators of finite multiplicative order in $M_n(S)$}
In this section,  we provide a sufficient condition for the existence of a commutator of finite multiplicative order in the matrix ring $M_n(S)$, the ring of all $n\times n$ matrices with entries from a unital ring $S$.

For an integer $k$, by $\overline{k}$ we denote the non-zero remainder when $k$ is divided by $n$.  Let $P,D \in M_n(S)$ be defined as follows :
\begin{equation}\label{definition of P and D}
P = \begin{pmatrix}
0      & 0      & 0      & \cdots & 1 \\
1      & 0      & 0      & \cdots & 0\\
0      & 1      & 0      &  \cdots & 0  \\
\vdots &        & \ddots & \ddots &  \vdots \\
0      & \cdots &        & 1      & 0
\end{pmatrix}
\;\;,\;\;
\quad
D = \begin{pmatrix} 
0 & 0 & 0 & \dots & 0 \\ 
0 & 1 & 0 & \dots & 0 \\ 
0 & 0 & 2 & \dots & 0 \\ 
\vdots & \vdots & \vdots & \ddots & \vdots \\ 
0 & 0 & 0 & \dots & n-1 
\end{pmatrix}\;.    
\end{equation}
\begin{lemma}\label{lem:PD}
 Let $P,D\in M_n(S)$ be as in \eqref{definition of P and D}.  Then, 
 \begin{equation*}
   [D,P]^n = (1-n)\;\Id_n \;. 
 \end{equation*}
\end{lemma}
 \begin{proof}
The $(i,j)$-th entries of $D$ and $P$ are as follows :
\begin{equation}\label{entrywise definition of D and P}
  (D)_{ij}=\begin{cases}
    i-1\;\;\;\text{if\;} i=j\\
    0\;\;\;\;\;\;\;\text{if}\; i\neq j\end{cases}\;,\;
(P)_{ij}=\begin{cases}
    1\;\;\;\text{if\;} i=\overline{j+1}\\
    0\;\;\;\text{if}\; i\neq \overline{j+1}\end{cases}\;.  
\end{equation}
Therefore, for $\Delta=\mathrm{diag}(1,\ldots,1,1-n)$, we have:
\begin{equation}\label{[D,P] is P times Delta}
\begin{split}
    ([D,P])_{ij}=(DP)&_{ij}-(PD)_{ij}=\sum_{k=1}^n D_{ik}P_{kj}-\sum_{k=1}^n P_{ik}D_{kj}\\&=\begin{cases}
        (i-1)-(j-1) \;\;\text{if\;} i=\overline{j+1}\\
        0\;\;\;\;\;\;\;\;\;\;\;\;\;\;\;\;\;\;\;\;\;\;\;\;\;\text{if\;}i\neq\overline{j+1}
    \end{cases}\\
    &= \begin{cases}
        1\;\;\;\;\;\;\;\;\;\text{if\;} i=\overline{j+1} \text{\;and\;} j = 1,\ldots,n-1\\
        1-n\;\;\text{if\;} i=\overline{j+1} \text{\;and\;} j=n \\0\;\;\;\;\;\;\;\;\;\text{if\;}i\neq\overline{j+1}
    \end{cases}=(P\Delta)_{ij}\;.
\end{split}
\end{equation}
By \eqref{entrywise definition of D and P}, it is easy to see that for each integer $k$, 
\begin{equation}\label{P power k}
    (P^k)_{ij}=\begin{cases}
    1\;\;\;\text{if\;} i=\overline{j+k}\\
    0\;\;\;\text{if}\; i\neq \overline{j+k}\end{cases}\,.
\end{equation}
Therefore, by \eqref{P power k},
\begin{equation}\label{P power k times Delta}
(P^k\Delta)_{ij}=\sum_{\ell=1}^n (P^k)_{i\ell}\;\Delta_{\ell j}
    = \begin{cases}
        \Delta_{jj} \;\;\text{if\;}i=\overline{j+k}\\
        0\;\;\;\;\;\;\text{if\;}i\neq\overline{j+k}
    \end{cases}\,.   
\end{equation}
Therefore, for $\Delta^{(k)} := P^k \Delta P^{-k}$, by \eqref{P power k times Delta}, 
\begin{equation}\label{entrywise Delta power (k)}
\begin{split}
(\Delta^{(k)})_{ij}&=\sum_{\ell=1}^n (P^k\Delta)_{i\ell}\;(P^{-k})_{\ell j}\\
    &= \begin{cases}
        {\Delta}_{\,\overline{j-k}\;\;\,\overline{j-k} } \;\;\text{if\;}i=\overline{\overline{j-k}+k}\\
        0\;\;\;\;\;\;\;\;\;\;\;\;\;\;\;\text{if\;}i\neq\overline{\overline{j-k}+k}
    \end{cases}
    =\begin{cases}
    {\Delta}_{\,\overline{i-k}\;\;\,\overline{i-k} }\;\;\;\; \text{if\;}i= j\\
    0\;\;\;\;\;\;\;\;\;\;\;\;\;\;\;\;\text{if\;}i\neq j
\end{cases}\;.
\end{split}
\end{equation}
As $P \Delta^{(k)} = \Delta^{(k+1)} P$, recursively we get:
\begin{equation}\label{P Delta whole power n}
    (P\Delta)^n = \Delta^{(1)} \Delta^{(2)} \cdots \Delta^{(n)}  P^n.
\end{equation}
As, by \eqref{entrywise Delta power (k)}, we have 
\begin{equation*}
    \big(\Delta^{(1)} \cdots \Delta^{(n)}\big)_{ij}
= \begin{cases}
    \prod_{k=1}^n {\Delta}_{\,\overline{i-k}\;\overline{i-k}}\;\;\text{if\;}i= j\\
    0\;\;\;\;\;\;\;\;\;\;\;\;\;\;\;\;\;\;\;\;\;\;\text{if\;}i\neq j
\end{cases}\;=\big((1-n)\;\Id_n\big)_{ij}\;,
\end{equation*}
and as $P^n=\Id_n$ (cf. \eqref{P power k}), by \eqref{[D,P] is P times Delta} and \eqref{P Delta whole power n}, the assertion follows.
\end{proof}

Lemma \ref{lem:PD} shows that the equality $[D,P]^n=\Id_n$ does not hold in an arbitrary $M_n(S)$.  We now observe under what circumstances the above equality can be satisfied, even by some generalised analogues of $D$ and $P$.  Denoting the set of units of $S$ and the center of $S$ of any unital ring $S$ by $U(S)$ and $Z(S)$ respectively, we have the following. 
\begin{theorem}\label{thm:main}
  Let $S$ be a unital ring and $n\geq 2$.  Suppose
  \begin{equation}\label{eq:cond}
    1 \;=\; v_1+v_2+\cdots+v_{n-1},
    \qquad v_1,\ldots,v_{n-1}\in U(Z(S)).
  \end{equation}
  Then there exist $A,B\in M_n(S)$ with $[A,B]^n=\Id_n$.
\end{theorem}  
\begin{proof}
We set $u_0=1$ and $u_k=-v_k$ for $1\leq k \leq n-1$.  Then, $u_k\in U(Z(S))$. Further, we set $s_0=0$ and $s_k=\sum_{j=0}^{k-1} u_j$, $1\leq k \leq n-1$.   We now choose $A,B\in M_n(S)$ as follows :
  \[
    A := \mathrm{diag}(s_0,\ldots,s_{n-1})
       = \begin{pmatrix}s_0&&\\&\ddots&\\&&s_{n-1}\end{pmatrix},
    \,
    B := \sum_{k=1}^{n}u_{k-1}^{-1}E_{\overline{k+1},\,k}
       = \begin{pmatrix}0&\cdots &0&u_{n-1}^{-1}\\u_0^{-1}&\cdots&0&0\\\vdots&\ddots&\vdots&\vdots\\0&\cdots &u_{n-2}^{-1}&0\end{pmatrix}.
  \]
That is, $(i,j)$-th entries of $A$ and $B$ are as follows :
$$(A)_{ij}=\begin{cases}
    s_{i-1}\;\;\;\text{if\;} i=j\\
    0\;\;\;\;\;\;\;\text{if}\; i\neq j\end{cases}\;,\;
(B)_{ij}=\begin{cases}
    u_j^{-1}\;\;\;\text{if\;} i=\overline{j+1}\\
    0\;\;\;\;\;\;\;\text{if}\; i\neq \overline{j+1}\end{cases}\;.$$
Therefore,
\begin{equation*}
\begin{split}
    ([A,B])_{ij}=(AB)&_{ij}-(BA)_{ij}=\sum_{k=1}^n A_{ik}B_{kj}-\sum_{k=1}^n B_{ik}A_{kj}\\&=\begin{cases}
        s_{\overline{j}}u_j^{-1}-u_j^{-1}s_{\overline{j-1}} \;\;\text{if\;} i=\overline{j+1}\\
        0\;\;\;\;\;\;\;\;\;\;\;\;\;\;\;\;\;\;\;\;\;\;\;\;\;\text{if\;}i\neq\overline{j+1}
    \end{cases}\\
    &= \begin{cases}
        1\;\;\text{if\;} i=\overline{j+1} \;\;(\text{as\;}u_j\in Z(S) \text{\;and\;}s_{\overline{j}}-s_{\overline{j-1}}=u_j) \\
        0\;\;\text{if\;}i\neq\overline{j+1}
    \end{cases}=(P)_{ij} \;(\text{by\;} \eqref{entrywise definition of D and P})\;.
\end{split}
\end{equation*}
  The assertion now follows from \eqref{P power k}.
\end{proof} 
\begin{corollary}\label{cor:n2}
  For any unital ring $S$, there exist $A,B\in M_2(S)$ with $[A,B]^2=\Id_2$.
\end{corollary}
 \begin{proof}
  The condition~\eqref{eq:cond} becomes a tautology by taking $v_1=1$.
\end{proof}
\begin{corollary}\label{cor:n3}
  Let $S$ be a unital ring.  If $1\in U(Z(S))+U(Z(S))$, i.e., there
  exists $u\in U(Z(S))$ with $1-u\in U(Z(S))$, then there exist
  $A,B\in M_3(S)$ with $[A,B]^3=\Id_3$.
\end{corollary}
 \begin{proof}
 The condition~\eqref{eq:cond} is satisfied by taking $v_1=u$ and $v_2=1-u$. 
\end{proof}
 \begin{corollary}\label{cor:general}
  Let $S$ be a unital ring and $n\geq 2$.  Each of the following
  conditions imply that there exist $A,B\in M_n(S)$ with $[A,B]^n=\Id_n$:
  \begin{enumerate}
    \item  $n-1\in U(Z(S))$.
    \item $\mathrm{char}(S)$ divides $(n-2)$.
  \end{enumerate}
\end{corollary}
\begin{proof}
    \begin{enumerate}
        \item The condition~\eqref{eq:cond} is satisfied by taking $v_k=(n-1)^{-1}$ for all $1\leq k \leq n-1$.
        \item The condition~\eqref{eq:cond} is satisfied by taking $v_k=1$ for all $1\leq k \leq n-1$.
    \end{enumerate}
\end{proof}
\begin{remark}
   The hypotheses of Corollary \ref{cor:n2}, Corollary \ref{cor:n3} and the first part of Corollary \ref{cor:general} are automatically true for $S=\mathbb{C}$ and therefore the conclusions hold.  In fact, for any $n\geq 2$, there exists $A,B\in M_n(\mathbb{C})$ satisfying $[A,B]^n=\Id_n$ by Theorem \ref{Main theorem_existential}. 
\end{remark}
\section{Structure of rings having a commutator of finite multiplicative order}
 
In this section, we investigate whether the existence of $a,b$ in a unital ring $R$, with $[a,b]^n=1$, gives us some structural results on $R$. Under some mild conditions on $R$, we are able to extract orthogonal idempotents whose sum is $1$. If these idempotents are equivalent (cf. Definition \ref{defn_cyclic equivalence}), then we prove that the ring $R$ is isomorhic to $M_n(S)$. We provide an equivalent condition for these idempotents to be equivalent. 

Let $\omega$ denote the $n$-th root of unity.  To prove the results in this section, we assume that $n \in U(R)$, $\omega\in U(Z(R))$ and $\omega^i-\omega^j \in U(R)$. An \textit{idempotent} $e \in R$ is an element such that $e^2=e$. Two idempotents $e_1$ and $e_2$ are said to be \textit{orthogonal} if $e_1e_2=e_2e_1=0$.

\begin{lemma} \label{lem: Creating_idempotents}
    Let $R$ be a unital ring with $n \in U(R)$, $\omega \in U(Z(R))$ and $\omega^i-\omega^j \in U(R)$. Suppose there exist $a,b \in R$ with $u:=[a,b]$ and $u^n=1$.

    Then there exist $e_0,e_1,\dots,e_{n-1}$ such that
        \begin{enumerate}
            \item $e_k^2=e_k$, $e_ke_{\ell}=0$ if $k \neq \ell$.
            \item $\sum_{k=0}^{n-1} e_k=1$.
            \item $u=\sum_{k=0}^{n-1} \omega^ke_k$.
        \end{enumerate}
    \end{lemma}
    \begin{proof}
        \begin{enumerate}
            \item Define $e_k:=\frac{1}{n}\sum_{j=0}^{n-1} \omega^{-kj}u^j$. 
            \begin{align*}
                e_ke_{\ell} &=\frac{1}{n^2}\sum_{r,s=0}^{n-1} \omega^{-kr-\ell s} u^{r+s}\\
                &=\frac{1}{n^2}\sum_{t=0}^{n-1}\Big(\sum_{r=0}^{n-1} \omega^{-kr-\ell(t-r)} \Big) u^t~~~ (\text{setting } t=(r+s) \mod n) \;.
            \end{align*}
            Note that
            \begin{align*}
                \sum_{r=0}^{n-1} \omega^{-(k-\ell)r} &=\begin{cases}
                    n & \text{if\;}k=l\\
                    0 & \text{if\;}k \neq l
                \end{cases}\;.
            \end{align*}
            Thus, if $k \neq \ell$, $e_ke_{\ell}=0$. If $k=\ell$, 
            \[
                e_k^2
                =\frac{1}{n^2}\sum_{t=0}^{n-1} n\omega^{-kt} u^t
                =\frac{1}{n} \sum_{t=0}^{n-1} \omega^{-kt}u^t
                =e_k.
            \]
            \item
            \[
                \sum_{k=0}^{n-1} e_k
                =\sum_{k=0}^{n-1}\Big( \frac{1}{n} \sum_{j=0}^{n-1} \omega^{-kj} u^j\Big)
                =\frac{1}{n} \sum_{j=0}^{n-1}\Big( \sum_{k=0}^{n-1}\omega^{-kj}\Big) u^j\;.
            \]
Set $S_j:=\sum_{k=0}^{n-1} \omega^{-kj}$. So, $S_0=n$. If $j \neq 0$, then
            $S_j=\frac{1-(\omega^{-j})^n}{1-\omega^{-j}}=0$.
\[
                \sum_{k=0}^{n-1} e_k=\frac{1}{n} \sum_{j=0}^{n-1} S_ju^j
                =\frac{1}{n} nu^0
                =1.
            \]
\item 
            \[\sum_{k=0}^{n-1} \omega^k e_k =\sum_{k=0}^{n-1}\omega^k \Big(\frac{1}{n} \sum_{j=0}^{n-1} \omega^{-kj}u^j\Big)
            =\frac{1}{n} \sum_{j=0}^{n-1} \sum_{k=0}^{n-1} \omega^{k(1-j)}u^j.\]
Set $T_j:=\sum_{k=0}^{n-1}\omega^{k(1-j)}$. If $j=1$, $T_1=n$. If $j \neq 1$, then $T_j=\frac{1-(\omega^{1-j})^n}{1-\omega^{1-j}}=0$.
            \[
                \sum_{k=0}^{n-1} \omega^k e_k =\frac{1}{n} \sum_{j=0}^{n-1} T_j u^j
                =\frac{1}{n}T_1u^1
                =\frac{1}{n}nu
                =u.
            \]
    \end{enumerate}
\end{proof}

Let $e$ be an idempotent in $R$. Then there exists a \textit{Pierce decomposition} of the ring $R=eRe \oplus eR(1-e) \oplus (1-e)Re \oplus (1-e)R(1-e)$. For $n$ orthogonal idempotents $e_0,e_1,\dots,e_{n-1}$ with $\sum_{i=0}^{n-1} e_i=1$, the Pierce decomposition is $R =\oplus_{i,j=1}^n e_iRe_j$. Note that all of these are left ideals of $R$. For a detailed discussion about Pierce decomposition, see \cite[\S 21]{Lam}. 
\begin{definition}\label{defn_cyclic equivalence}
    A set of orthogonal idempotents $e_0,e_1,\dots,e_{n-1}$ of a ring $R$ are called \textit{cyclically equivalent} if  for all $1 \leq i \leq n-1$, with $e_n=e_0$, there exist $x_i \in e_iRe_{i+1}$ and $y_i \in e_{i+1}Re_i$ such that
    \begin{align*}
        x_iy_i=e_i \;,\; y_ix_i=e_{i+1}\;.
    \end{align*}
\end{definition}

Throughout the article, we write necessary expression in terms of $e_n$, where $e_n=e_0$.

\begin{proposition} \label{prop:cyclically equivaence}
    Let $R$ be a unital ring with $n \in U(R)$, $\omega \in U(Z(R))$ and $\omega^i-\omega^j \in U(R)$. Suppose there exist $a,b \in R$ with $u:=[a,b]$ and $u^n=1$.

    Let $e_0,e_1,\dots,e_{n-1}$ be idempotents as mentioned in Lemma \ref{lem: Creating_idempotents}. Then the following are equivalent:
    \begin{enumerate}[(i)]
        \item The idempotents $e_0,e_1,\dots,e_{n-1}$ are cyclically equivalent.
        \item There exists $v \in U(R)$ such that $vuv^{-1}=\omega^{-1}u$.
    \end{enumerate}
\end{proposition}
\begin{proof}
    (\textbf{$\Rightarrow$}) Suppose $e_0,e_1,\dots,e_{n-1}$ are cyclically equivalent. Then there exist $x_k \in e_kRe_{k+1}$ and $y_k \in e_{k+1}Re_k$ such that
    \[x_ky_k=e_k\;,\;y_kx_k=e_{k+1}.\]
    Let $w:=\sum_{i=0}^{n-1} x_i$ and $v:=\sum_{j=0}^{n-1}y_j$. Note that $x_iy_j=0$ if $i \neq j$ as $x_i \in e_iRe_{i+1}$ and $y_i \in e_{j+1}Re_{j}$ and hence $vw=\sum_{i=0}^{n-1} y_ix_i=\sum_{i=0}^{n-1} e_{i+1}=1$ and $wv=\sum_{i=0}^{n-1} x_iy_i=\sum_{i=0}^{n-1} e_i=1$. Thus, $v$ is invertible and $v^{-1}=w$.
    \[ve_kv^{-1}
    =\left(\sum_{i=0}^{n-1} y_i\right)e_k \left(\sum_{i=0}^{n-1} x_i\right)
    =y_ke_k\left(\sum_{i=0}^{n-1} x_i\right)
    =y_ke_kx_k
    =y_kx_k
    =e_{k+1}.
    \]
    Therefore, 
    $vuv^{-1}=v\big(\sum_{k=0}^{n-1} \omega^{k} e_k\big) v^{-1}
    =\sum_{k=0}^{n-1} \omega^k e_{k+1}
    =\omega^{-1} \big( \sum_{k=0}^{n-1} \omega^{k+1} e_{k+1}\big)
    =\omega^{-1}u$.

    (\textbf{$\Leftarrow$}) Assume that there exists $ v \in U(R)$ such that $vuv^{-1}=\omega^{-1}u$. Let
    \begin{align*}
        p_k(t):=\prod_{j \neq k} \frac{t-\omega^j}{\omega^k-\omega^j}\;.
    \end{align*}
    Note that $p_k(\omega^j)=\delta_{kj}$, where $\delta_{ij}$ is the Kronecker delta.  Since $e_k$s are orthogonal idempotents,   
    \[p_k(u) =\sum_{i=0}^{n-1}p_k(\omega^i)e_i=\sum_{i=0}^{n-1}\delta_{ki} e_i=e_k. \]
    Therefore,
    \begin{align*}
        ve_kv^{-1} &=vp_k(u)v^{-1}\\
        &=p_k(vuv^{-1})\\
        &=p_k(\omega^{-1}u)\\
        &=p_k\left(\sum_{i=0}^{n-1}\omega^{i-1}e_i\right)\\
        &=\sum_{i=0}^{n-1}\delta_{k\;(i-1)}e_i\\
        &=e_{k+1}.
    \end{align*}
    Choose $y_i=ve_k$ and $x_i=e_kv^{-1}$. Then $x_iy_i=e_k$ and $y_ix_i=ve_kv^{-1}=e_{k+1}$.
    
\end{proof}
\begin{remark}
    Even if we assume that $\omega \in U(Z(R))$, it does not give us $\omega^i-\omega^j \in U(R)$. For example, take $R=\ZZ/8\ZZ$, and $\omega=3+8\ZZ$, then $\omega^2=1+8\ZZ$ and hence $\omega$ is a second root of unity. But $\omega-(1+8\ZZ)=2+8\ZZ$ is not invertible in $R$, as $(2+8\ZZ)(4+8\ZZ)=0+8\ZZ$.
\end{remark}
\begin{remark}
    Assuming the ring $R$ has characteristic zero also does not help. For example, in $Z[\omega]$, where $\omega$ is the $n$-th root of unity,  $\omega^i-\omega^j$ is non-zero but not a unit in most of the cases.  Consider $\ZZ[\mathrm{i}]$, where ${\mathrm{i}}^2=-1$, and therefore $\mathrm{i}$ is a $4$-th root of unity. In $\ZZ[\mathrm{i}]$, $(\mathrm{i}-1)$ is not invertible: if $(\mathrm{i}-1)$ is invertible, then there exists $(a+\mathrm{i}b) \in \ZZ[\mathrm{i}]$ such that $(\mathrm{i}-1)(a+\mathrm{i}b)=1$, gives $a=-\frac{1}{2}\notin \ZZ$.
\end{remark}

We note that if $R$ has equivalent orthogonal idempotents that sum up to $1$, the structure of $R$ gets determined.
\begin{proposition} \label{prop: Ring structure with orthogonal idemponet sum 1}
    Let $R$ be a unital ring. Suppose there exist idempotents $e_0,e_1,\dots,e_{n-1} \in R$ such that
    \begin{enumerate}[(i)]
        \item $e_k^2=e_k$ and $e_ie_j=0 \text{ for } i \neq j$.
        \item $\sum_{i=0}^{n-1} e_i=1$.
        \item $e_0,e_1,\dots,e_{n-1}$ are cyclically equivalent.
    \end{enumerate}
    Then $R \cong M_n(S)$, where $S=e_0Re_0$.
\end{proposition}
\begin{proof}
    For $i=0,1,\dots,n-1$, let $x_i\in e_iRe_{i+1}, y_i \in e_{i+1}Re_i$ be such that $x_iy_i=e_i$ and $y_ix_i=e_{i+1}$.\\
For $i \geq 1$, let
    \begin{align*}
        X_i:=y_{i-1}y_{i-2}\dots y_0 \in e_iRe_0\;\text{\;and\;}\;Y_i:=x_0x_1 \dots x_{i-1} \in e_0Re_i\;.
    \end{align*}
    and $X_0=Y_0=e_0$.
    \begin{equation*}\label{product_XiYi}
      \begin{split}
        X_iY_i&=(y_{i-1}y_{i-2} \dots y_0)(x_0x_1 \dots x_{i-1})\\
        &=y_{i-1}y_{i-2} \dots y_{1}(y_0x_0)x_1 \dots x_{i-1}\\
        &=y_{i-1}y_{i-2} \dots y_{1}e_1x_1 \dots x_{i-1}\\
        &=y_{i-1}y_{i-2} \dots (y_1x_1) \dots x_{i-1}~~(\text{since } e_1x_1=x_1 \text{ as } x_1  \in e_1Re_2)\\
        &\;\;\;\;\;\vdots\\
        &=y_{i-1}x_{i-1}\\
        &=e_i.
    \end{split}
      \end{equation*}
    By similar computation,
\begin{equation}\label{product of X_i and Y_j}
    Y_jX_k=\begin{cases}
        e_0 &\text{if\;}j=k\\
        0 & \text{if\;}j \neq k
    \end{cases}\;.
\end{equation}
 Note that for any $r \in R$,  
    \begin{equation}\label{YirXj is in S}
        Y_irX_j \in (e_0Re_i)R(e_jRe_0) \subseteq e_0Re_0=S.
    \end{equation}
    Define
    \begin{align*}
        \phi: & M_n(S) \to R\\
        & (s_{ij}) \to \sum_{i,j}X_is_{ij}Y_j\;.
    \end{align*}
    This $\phi$ is a ring homomorphism.
    
    \textbf{Injectivity of $\phi$:} Let $(s_{ij}) \in M_n(S)$ such that $\phi\Big((s_{ij})\Big)=0$, that is $\sum_{i,j} X_i s_{ij} Y_j=0$.  This implies $\sum_{i,j} Y_kX_i s_{ij} Y_jX_l=0$. Hence, by \eqref{product of X_i and Y_j}, $(Y_iX_i) s_{ij} (Y_jX_j)=e_0 s_{ij} e_0=0$, for $0 \leq i,j \leq n-1$.  Hence, $s_{ij}=0$.
    
    \textbf{Surjectivity of $\phi$:} Let $r \in R$, then $r=\sum_{i,j} e_ire_j=\sum_{i,j} X_iY_irX_jY_j$. Choose $s_{ij}=Y_irX_j \in e_0Re_0$ (cf. \eqref{YirXj is in S}), then $r=\sum_{i,j} X_i s_{ij} Y_j=\phi((s_{ij}))$.
\end{proof}

Before proving the final result, we note the following proposition to emphasize the fact that the assumption $\omega^i-\omega^j \in U(R)$ is quite natural, wherever it appears.
\begin{proposition}
    Let $R$ be a unital ring with $2 \in U(R)$. Suppose there exist $a,b \in R$ with $u:=[a,b]$ and $u^2=1$.

    Let $e_0=\frac{1+u}{2}$ and $e_1=1-e_0=\frac{1-u}{2}$. Then the following are equivalent:
\begin{enumerate}[(i)]
        \item There exist $x, y \in R$ such that
        $xy=e_0$ and $yx=e_1$.
        \item There exist $x,y \in R$ such that $[x,y]=u$ and $xy+yx=1$.
    \end{enumerate}
\end{proposition}
\begin{proof}
    \big(\textbf{$(i)\Rightarrow (ii)$}\big) Let $xy=e_0$ and $yx=e_1$. Then $xy+yx=e_0+e_1=1$ and $xy-yx=e_0-e_1=\frac{1+u}{2}-\frac{1-u}{2}=u$.

    \big(\textbf{$(ii)\Rightarrow (i)$}\big) Suppose there exist $x,y$ such that $[x,y]=u$ and $xy+yx=1$. Then $2xy=1+u$, gives $xy=\frac{1+u}{2}=e_0$. Similarly $2yx=1-u$, gives $yx=\frac{1-u}{2}=e_1$.
\end{proof}
\begin{remark}
    Note that in this case, i.e., when $n=2$ and $w=-1$ is a $2$-nd root of unity, $w^0-w^1=1-w=2\in U(R)$.  That is, the condition $2\in U(R)$ automatically forces $\omega^i-\omega^j \in U(R)$.
\end{remark}
\begin{theorem}\label{thm: Structure theorem due to commutator}
    Let $R$ be a unital ring with $n \in U(R)$, $\omega \in U(Z(R))$ and $\omega^i-\omega^j \in U(R)$  for all $i \neq j,~ 0 \leq i,j \leq n-1$. Suppose there exist $a,b \in R$ with $u:=[a,b]$ and $u^n=1$.

    If there exists $v \in U(R)$ such that $vuv^{-1}=\omega^{-1}u$, then 
    $R \cong M_n(S)$, where $S=e_0Re_0$, with $e_0$ an idempotent in $R$.
\end{theorem}
\begin{proof}
    By Lemma \ref{lem: Creating_idempotents}, there exist orthogonal idempotents $e_0,e_1,\dots e_{n-1}$ with $\sum_{i=0}^{n-1} e_i=1$. By Proposition \ref{prop:cyclically equivaence}, all $e_i$ are cyclically equivalent. By Proposition \ref{prop: Ring structure with orthogonal idemponet sum 1}, $R \cong M_n(S)$, where $S=e_0Re_0$.
\end{proof}

Let $R$ be a unital ring.  By $\xi(R)$, we denote the smallest positive integer $N$ such that every element of $R$ can be written as a sum of $N$ many terms, where each term is a product of two commutators (cf. \cite[Definition 5.4, p.~230]{GaTh}).  This is an important invariant and has been studied by several researchers (cf. \cite{GaTh}, \cite{Me} and \cite{Ro}).  As an immediate consequence of Theorem \ref{thm: Structure theorem due to commutator}, we obtain an upper bound of the invariant $\xi(R)$ for some unital ring $R$. 


\begin{corollary}
Let $R$ be a unital ring with $n \in U(R)$, $\omega \in Z(U(R))$ a primitive $n$-th
root of unity, and $\omega^{i} - \omega^{j} \in U(R)$ for all $i \neq j,~ 0 \leq i,j \leq n-1$.
If there exists $v \in U(R)$ such that $vuv^{-1}=\omega^{-1}u$,
\[
    \xi(R) \leq 2. 
\]
\end{corollary}

\begin{proof}
Under the given hypotheses, Theorem \ref{thm: Structure theorem due to commutator} gives $R \cong M_n(S)$ for some unital ring $S$.
The bound $\xi(M_n(S)) \leq 2$ then follows directly from \cite[Theorem~5.4, p.~231]{GaTh}.
\end{proof}

As another application of Theorem \ref{thm: Structure theorem due to commutator}, we provide an example in which a subring of a polynomial ring turns out to be a matrix ring.
\begin{example}
    Let $K=\mathbb{Q}(\omega)$  and \[R=\{K[x,y] \ | \ x^n=a,y^n=b,yx=\omega xy\},\] where $a,b \in \QQ(\omega)$.
    Note that, as $yx=\omega xy$ holds in $R$, the following equality holds as well in $R$ : 
\begin{equation}\label{interchange}
  y^nx=\omega^n xy^n.  
\end{equation}
We claim that $(xy)^n=w^{^n C_2}x^ny^n$. We prove it by induction.  For $n=2$, we have 
\begin{equation*}(xy)^2=xyxy=x\omega xyy=\omega x^2y^2=\omega^{^2 C_2}x^2y^2.
\end{equation*}
Assume that our claim is true for $n-1$. Then, 
\begin{align*}
(xy)^n=&(xy)^{n-1}xy\\
=&w^{^{n-1} C_2}x^{n-1}y^{n-1}xy\\
=&w^{^{n-1} C_2}x^{n-1}\omega^{n-1}xy^{n-1}y\quad (\text{by}\quad \eqref{interchange})\\
=&w^{\left({^{n-1} C_2}+(n-1)\right)}x^ny^n\\
=&w^{^nC_ 2}x^ny^n
\end{align*}
Hence, the claim follows.\\
Now, choose  $a,b \in \QQ(\omega)$ such that
    \[(1-\omega)^n \omega^{^nC_2}ab=1.\]
    Note that $x$ is invertible in $R$. Take $v=x$ and $u=[x,y]$. Then $u^n=(xy-yx)^n=(1-\omega)^n(xy)^n=(1-\omega)^n \omega^{^nC_2}x^ny^n=(1-\omega)^n \omega^{^nC_2}ab=1$. It is easy to note that $vuv^{-1}=\omega^{-1}u$. Therefore, by Theorem \ref{thm: Structure theorem due to commutator}, $R \cong M_n(S)$, where $S$ is a unital ring, in particular $S=e_0Re_0$ with $e_0=\frac{1}{n}(1+u+u^2+\dots+u^{n-1})=\frac{1}{n}(1+(1-\omega) xy+(1-\omega)^2(xy)^2+\dots+(1-\omega)^{n-1}(xy)^{n-1})$.
\end{example}
\subsection*{Alternative proof of Theorem \ref{thm: Structure theorem due to commutator}}
In the paper \cite{GoKh}, Goyal and Khurana have established an equivalent characterization for a ring to be a matrix ring in terms of a single element. We use the following result from \cite{GoKh} to provide an alternative proof of Theorem \ref{thm: Structure theorem due to commutator}. 

Recall that an element $x \in R$ is said to be \textit{regular} if there exists $y \in R$ such that $xyx=x$. Note that any invertible element $x$ is regular with $y=x^{-1}$. By  $l_R(x)$, we denote the set of \textit{left annihilators of} $x$, that is 
\[l_R(x):=\{a \in R\;\mid\; ax=0\}.\]

\begin{proposition}[Theorem 2.4, p.~97, \cite{GoKh}] \label{prop: Goyal Khurana characterization of matrix rings}
A ring $R$ is an $n \times n$ matrix ring if and only if there exists a regular element $x$ in $R$ such that $l_R(x) = Rx^{n-1}$. Moreover, if a regular element $x$ with $l_R(x) = Rx^{n-1}$ exists, then $R \cong M_n(\End_R(Rx^{n-1}))$.
\end{proposition}

We use the standing hypothesis on $R$ in Theorem \ref{thm: Structure theorem due to commutator} to show the existence of a regular element $x$ satisfying the left annihilator condition $l_R(x)=Rx^{n-1}$.

\begin{proposition}\label{prop: exitence of regular element}
    Let $R$ be a unital ring with $n \in U(R)$, $\omega \in U(Z(R))$ and $\omega^i-\omega^j \in U(R)$  for all $i \neq j,~ 0 \leq i,j \leq n-1$. Suppose there exist $a,b \in R$ with $u:=[a,b]$ and $u^n=1$.

    There exists $v \in U(R)$ such that $vuv^{-1}=\omega^{-1}u$.
    Then there exists a regular element $x$ with $l_R(x)=Rx^{n-1}$.
\end{proposition}

\begin{proof}
    By Lemma \ref{lem: Creating_idempotents}, there exist orthogonal idempotents $e_0,e_1,\dots, e_{n-1}$ with $\sum_{i=0}^{n-1} e_i=1$. By Proposition \ref{prop:cyclically equivaence}, all $e_i$ are cyclically equivalent, in particular
    \begin{equation}\label{cy eqv consequence}
     ve_kv^{-1}=e_{k+1}\;,   
    \end{equation}
    where the indices are considered $\mod\;n$.

    Set $x=v(1-e_{n-1})$. We claim that $x$ is regular and satisfy the left annihilator condition $l_R(x)=Rx^{n-1}$.

    \textbf{Regularity:} Let $y:=(1-e_{n-1})v^{-1}$.
    \begin{align*}
        xyx&=v(1-e_{n-1})(1-e_{n-1})v^{-1}v(1-e_{n-1})\\
        &=v(1-e_{n-1})\\
        &=x\;.
    \end{align*}

    \textbf{Annihilator condition:} We claim that for any $m, 1 \leq m \leq n-1$,
    \[x^m=v^m(1-e_{n-m}-\dots-e_{n-1})\;.\]
    We prove it by induction. The base case $m=1$ follows from the definition of $x$. We assume that it is true for any $1 \leq m <n-1$. We prove it for $m+1$. Since for any $m, 1 \leq m \leq n-1$, we have $e_jv=ve_{j-1}$ by \eqref{cy eqv consequence}, the following equality follows :
\begin{align}\label{eq:intertwining_v}
    (1-e_{n-m}-\dots-e_{n-1})v=v(1-e_{n-m-1}-\dots-e_{n-2}).
    \end{align}
    We prove the induction step.
    \begin{align*}
        x^{m+1}&=x^mx\\
        &=v^m(1-e_{n-m}-\dots-e_{n-1})x\;\;\;\;\;\;\;\;\;\;\;\;\;\;\;\;\;\;\;\;\;\big(\text{by induction hypothesis}\big)\\
        &=v^m(1-e_{n-m}-\dots-e_{n-1})v(1-e_{n-1})\\
        &=v^mv(1-e_{n-m-1}-\dots-e_{n-2})(1-e_{n-1})\;\;\;\;\;\;\;\;\;\big(\text{by } \eqref{eq:intertwining_v}\big)\\
        &=v^{m+1}(1-e_{n-m-1}-\dots-e_{n-2}-e_{n-1})\;\;\;\;\;\;\big(\text{by orthogonality of }e_k s\big)
    \end{align*}
    Thus $x^m=v^m(1-e_{n-m}-\dots-e_{n-1})$ for all $1\leq m \leq n-1$. In particular for $m=n-1$, we get $x^{n-1}=v^{n-1}(1-e_1-e_2-\dots-e_{n-1})=v^{n-1}e_0$. Since $r \mapsto rv^{n-1}$ is a bijection of $R$, 
    $$Rx^{n-1}=Rv^{n-1}e_0=Re_0.$$
    Therefore, for any $a \in R$
    \begin{align*}
        ax=0 &\iff av(1-e_{n-1})=0\\
        & \iff av=ave_{n-1}\\
        & \iff av \in Re_{n-1}\\
        & \iff a \in Re_{n-1}v^{-1}\\
        & \iff a \in Rv^{-1}e_0 \;\;\;\;\;\;\;\;\;\;\big(\text{since } e_{n-1}v^{-1}=v^{-1}e_0 \text{ by } \eqref{cy eqv consequence}\big)\\
        & \iff a \in Re_0.
    \end{align*}
    Therefore, $l_R(x)=Re_0=Rx^{n-1}$.
\end{proof}

\begin{theorem}
    Let $R$ be a unital ring with $n \in U(R)$, $\omega \in U(Z(R))$ and $\omega^i-\omega^j \in U(R)$  for all $i \neq j,~ 0 \leq i,j \leq n-1$. Suppose there exist $a,b \in R$ with $u:=[a,b]$ and $u^n=1$.

    There exists $v \in U(R)$ such that $vuv^{-1}=\omega^{-1}u$.
    Then there exists a ring $S$ such that $R \cong M_n(S)$.
\end{theorem}

\begin{proof}
    Using Proposition \ref{prop: exitence of regular element} we get a regular element $x \in R$ such that $l_R(x)=Rx^{n-1}$. Then by Proposition \ref{prop: Goyal Khurana characterization of matrix rings} $R \cong M_n(S)$, where $S=\End_R(Rx^{n-1})$.
\end{proof}
\begin{remark}
It should be noted that $Rx^{n-1}$ is isomorphic to $Re_0$ and $\text {End}(Re_0)\cong e_0Re_0$ as rings (cf. \cite[Corollary 21.7, p.~320]{Lam}). Therefore, Theorem \ref{thm: Structure theorem due to commutator} is consistent with Proposition \ref{prop: Goyal Khurana characterization of matrix rings}.
\end{remark}

\section*{Acknowledgements}
   The authors express their gratitude to Prof. D Khurana for making the authors aware of the paper \cite{GoKh}.  The authors would also like to acknowledge A Naolekar, P W Ng and B Sury for their suggestions regarding this article.  The first-named author would like to thank the Indian Institute of Technology Madras for financial support (Office order No.F.ARU/R10/IPDF/2024). The third-named author gratefully acknowledges the financial support and hospitality provided by Dr. S Barik during the visit at ISI Bangalore, which is funded through his INSPIRE Faculty Fellowship Research Grant (DST/INSPIRE/Faculty/\\2023/IFA-23-MA-201) from the Department of Science and Technology (DST), Government of India.  

\end{document}